\documentclass[12pt]{amsart}
\advance\hoffset by -0.5 truecm
\usepackage{amssymb}
\newtheorem{Theorem}{Theorem}[section]
\newtheorem{Lemma}[Theorem]{Lemma}

\newtheorem{Definition}[Theorem]{Definition}
\newtheorem{Remark}[Theorem]{Remark}

\def\V{\mbox{Var}}

\def\Z{{\mathbb Z}}
\def\R\re
\def\V{\bf V}

\def \re{{\mathbb R}}

\def \C{{\mathbb C}}

\def \V{{\bf V}}

\def \M{{\mathbb M}}

\def \cA{{\mathcal A}}
\def \cH{{\mathcal H}}

\begin{document}
\title[B\"acklund transformations for Transparent connections]{B\"acklund transformations for Transparent connections}


\author[G.P. Paternain]{Gabriel P. Paternain}
 \address{ Department of Pure Mathematics and Mathematical Statistics,
University of Cambridge,
Cambridge CB3 0WB, UK}
 \email {g.p.paternain@dpmms.cam.ac.uk}




\begin{abstract} Let $M$ be a closed orientable surface of negative curvature.
A connection is said to be transparent if its parallel transport along closed
geodesics is the identity. We describe all transparent $SU(2)$-connections
and we show that they can be built up from suitable B\"acklund transformations.

\end{abstract}

\maketitle

\section{Introduction}

Let $(M,g)$ be a closed Riemannian manifold.
A unitary connection $\nabla$ on the trivial bundle $M\times\C^n$ is said to be {\it transparent} if its parallel transport along every closed
geodesic of $g$ is the identity. These connections are ``ghosts" or ``invisible" from the point of view of the closed geodesics of $g$. 
The problem of determining a connection from its parallel transport along geodesics is a natural integral-geometry problem
that can be considered also in the case of manifolds with boundary or $\re^d$ with appropriate decay
conditions at infinity, and in this context, it has been considered by several authors \cite{E,FU,No,Sha,V}. It arises for example when one considers the wave equation associated to the
Schr\"odinger equation with external Yang-Mills potential $A$ and the inverse problem of determining
the potential $A$ from the Dirichlet-to-Neumann map $\Lambda_{A}$. 
Returning to the case of closed manifolds we note  
that recently, L. Mason
has classified all transparent connections on $S^2$ and $\re \mathbb P^2$ with the standard round metric \cite{Ma1} using a twistor correspondence similar
to the one used in \cite{Ma} to classify anti-self-dual Yang-Mills connections over $S^2\times S^2$ with split signature.

There is another natural but stronger notion of transparency which arises by considering an appropriate cocycle over the geodesic flow. Write $\nabla=d+A$, where
$A:TM\to \mathfrak{u}(n)$ is a smooth function linear in $v\in T_{x}M$ for all $x\in M$.
Consider the geodesic flow $\phi_t$ of the metric $g$ acting on the unit sphere bundle
$SM$. The connection defines a $U(n)$-valued cocycle over the flow $\phi_t$ determined
by the ODE of parallel transport along geodesics. In other words, let
$C:SM\times \re\to U(n)$ be given by
\[\frac{d}{dt}C(x,v,t)=-A(\phi_{t}(x,v))C(x,v,t),\;\;\;\;\;C(x,v,0)=\mbox{\rm Id}.\]
The function $C$ is a {\it cocycle}: 
\[C(x,v,t+s)=C(\phi_{t}(x,v),s)\,C(x,v,t)\]
for all $(x,v)\in SM$ and $s,t\in\re$. The cocycle $C$ is said to be {\it cohomologically trivial}
if there exists a smooth function $u:SM\to U(n)$ such that
\[C(x,v,t)=u(\phi_{t}(x,v))u^{-1}(x,v).\]

\begin{Definition} {\rm We will say that a connection $\nabla$ is cohomologically trivial
if $C$ is cohomologically trivial.}
\end{Definition}

Obviously a cohomologically trivial connection is transparent, since the latter
simply means $C(x,v,T)=\mbox{\rm Id}$ every time that $\phi_{T}(x,v)=(x,v)$.
There is one important situation in which both definitions agree. If $\phi_t$ is
Anosov, then the Livsic theorem \cite{L1,L2} together with the regularity results in \cite{NT}
imply that a transparent connection is also cohomologically trivial. 
The Anosov property is satisfied, if for example $(M,g)$ has negative sectional curvature.
  
In \cite{P} we studied transparent connections over negatively curved surfaces
and established several of their basic properties. In the present paper we will explain
how to obtain all transparent connections when the structure group is $SU(2)$
(our main result is Theorem \ref{theorem:main} below).
The description of the connections is based on a B\"acklund transformation  which is quite
reminiscent of Uhlenbeck's work on unitons \cite{U}. At the same time we will set up a general
correspondence for cohomologically trivial connections over any surface regardless
of the sign of its Gaussian curvature thus generalising Theorem B in \cite{P}.

\medskip

\noindent{\it Acknowledgement:} I am very grateful to Maciej Dunajski for several useful discussions related to this paper. In particular, 
he suggested to me the possibility that a B\"acklund transformation may exist for transparent connections.

\section{Preliminary results}
\label{prelim}
In this section we summarise the setting in \cite{P} and the results needed for
the subsequent sections.

Let $M$ be an oriented surface with a Riemannian metric and let
$SM$ be its unit tangent bundle. Recall that $SM$ has a
canonical framing $\{X,H,V\}$, where $X$ is the geodesic vector field, $V$
is the vertical vector field and $H=[V,X]$ is the horizontal vector field.

Let $\M_{n}(\mathbb C)$ be the set of $n\times n$ complex
matrices. Given functions $u,v:SM\to \M_{n}(\mathbb C)$ we consider the
inner product
\[\langle u,v \rangle =\int_{SM}\mbox{\rm trace}\,(u\,v^*)\,d\mu,\]
where $\mu$ is the Riemannian measure associated with the Sasaki metric
of $SM$ which makes $\{X,H,V\}$ into an orthonormal frame.
The space $L^{2}(SM,\M_{n}(\mathbb C))$ decomposes orthogonally
as a direct sum
\[L^{2}(SM,\M_{n}(\mathbb C))=\bigoplus_{n\in\mathbb Z}H_{n}\]
where $-iV$ acts as $n\,\mbox{\rm Id}$ on $H_n$.

Following Guillemin and Kazhdan in \cite{GK} we introduce the following
first order elliptic operators 
$$\eta_{+},\eta_{-}:C^{\infty}(SM,\M_{n}(\mathbb C))\to
C^{\infty}(SM,\M_{n}(\C))$$ given by
\[\eta_{+}:=(X-iH)/2,\;\;\;\;\;\;\eta_{-}:=(X+iH)/2.\]
Clearly $X=\eta_{+}+\eta_{-}$. Let $\Omega_{n}:=C^{\infty}(SM,\M_{n}(\C))\cap H_{n}$.
We have
\[\eta_{+}:\Omega_{n}\to \Omega_{n+1},\;\;\;\;\eta_{-}:\Omega_{n}\to \Omega_{n-1},\;\;\;\;(\eta_{+})^{*}=-\eta_{-}.\]

If a connection $A$ is cohomologically trivial there exists a smooth $u:SM\to U(n)$ such that $C(x,v,t)=u(\phi_{t}(x,v))u^{-1}(x,v)$.
Differentiating with respect to $t$ and setting $t=0$, this is equivalent to
$X(u)+Au=0$, where we now regard $A$ as a function $A:SM\to\mathfrak{u}(n)$.
To deal with this equation, we introduce the ``twisted'' operators
\[\mu_{+}:=\eta_{+}+A_1,\;\;\;\;\mu_{-}:=\eta_{-}+A_{-1}.\]
where $A=A_{-1}+A_{1}$, and
\[A_{1}:=\frac{A-iV(A)}{2}\in H_{1},\]
\[A_{-1}:=\frac{A+iV(A)}{2}\in H_{-1}.\]
Observe that this decomposition corresponds precisely with the usual decomposition of $\mathfrak{u}(n)$-valued 1-forms
on a surface:
\[\Omega^{1}(M,\mathfrak{u}(n))\otimes \C=\Omega^{1,0}(M,\mathfrak{u}(n))\oplus \Omega^{0,1}(M,\mathfrak{u}(n)),\]
given by the eigenvalues $\pm i$ of the Hodge star operator $\star$ of the metric $g$.

We also have
\[\mu_{+}:\Omega_{n}\to \Omega_{n+1},\;\;\;\;\mu_{-}:\Omega_{n}\to \Omega_{n-1},\;\;\;\;(\mu_{+})^{*}=-\mu_{-}.\]
The equation $X(u)+Au=0$ is now $\mu_{+}(u)+\mu_{-}(u)=0$.

Given an element $u\in C^{\infty}(SM,\M_{n}(\C))$, we write
$u=\sum_{m\in\Z}u_{m}$, where $u_m\in \Omega_m$.
We will say that $u$ has degree $N$, if $N$ is the smallest
non-negative integer such that $u_{m}=0$ for all $m$ with $|m|\geq N+1$.
The following finiteness result will be crucial for us.

\begin{Theorem}\cite[Theorem 5.1]{P} If $M$ has negative curvature every solution
$u$ of $X(u)+Au=0$ has finite degree.
\label{theorem:finite}
\end{Theorem}


For future use, it is convenient to write the operators $\eta_{-}$ and $\mu_{-}$ in local
coordinates. Consider isothermal coordinates $(x,y)$ on $M$ such that the metric
can be written as $ds^2=e^{2\lambda}(dx^2+dy^2)$ where $\lambda$ is a smooth
real-valued function of $(x,y)$. This gives coordinates $(x,y,\theta)$ on $SM$ where
$\theta$ is the angle between a unit vector $v$ and $\partial/\partial x$.
In these coordinates, $V=\partial/\partial\theta$ and
the vector fields $X$ and $H$ are given by:
\[X=e^{-\lambda}\left(\cos\theta\frac{\partial}{\partial x}+
\sin\theta\frac{\partial}{\partial y}+
\left(-\frac{\partial \lambda}{\partial x}\sin\theta+\frac{\partial\lambda}{\partial y}\cos\theta\right)\frac{\partial}{\partial \theta}\right);\]
\[H=e^{-\lambda}\left(-\sin\theta\frac{\partial}{\partial x}+
\cos\theta\frac{\partial}{\partial y}-
\left(\frac{\partial \lambda}{\partial x}\cos\theta+\frac{\partial \lambda}{\partial y}\sin\theta\right)\frac{\partial}{\partial \theta}\right).\]
Consider $u\in\Omega_n$ and write it locally as $u(x,y,\theta)=h(x,y)e^{in\theta}$.
Using these formulas a simple, but tedious calculation shows that
\begin{equation}
\eta_{-}(u)=e^{-(1+n)\lambda}\bar{\partial}(he^{n\lambda})e^{i(n-1)\theta},
\label{eq:eta}
\end{equation}
where $\bar{\partial}=\frac{1}{2}\left(\frac{\partial}{\partial x}+i\frac{\partial}{\partial y}\right)$.
In order to write $\mu_{-}$ suppose that $A(x,y,\theta)=a(x,y)\cos\theta+b(x,y)\sin\theta$.
If we also write $A=A_{x}dx+A_{y}dy$, then $A_x=ae^{\lambda}$ and $A_y=be^{\lambda}$.
Let $A_{\bar{z}}:=\frac{1}{2}(A_{x}+iA_{y})$.
Using the definition of $A_{-1}$ we derive
\begin{equation}
A_{-1}=\frac{1}{2}(a+ib)e^{-i\theta}=A_{\bar{z}}d\bar{z}.
\label{eq:a1}
\end{equation}
Putting this together with (\ref{eq:eta}) we obtain
\begin{equation}
\mu_{-}(u)=e^{-(1+n)\lambda}\left(\bar{\partial}(he^{n\lambda})+A_{\bar{z}}he^{n\lambda}\right)e^{i(n-1)\theta}.
\label{eq:mu}
\end{equation}

Note that $\Omega_n$ can be identified with the set of smooth sections
of the bundle $(M\times \M_{2}(\mathbb C))\otimes K^{\otimes n}$ where $K$ is the canonical line bundle. The identification takes $u=he^{in\theta}$ into $he^{n\lambda}(dz)^n$ ($n\geq 0$)
and $u=he^{-in\theta}\in \Omega_{-n}$ into $he^{n\lambda}(d\bar{z})^n$.
The second equality in (\ref{eq:a1}) should be understood using this
identification.

\section{A general correspondence}

The purpose of this section is to briefly indicate that the calculations done in the
proof of Theorem B in \cite{P} also give a general classification result for cohomologically
trivial connections on any surface.

Let $M$ be an oriented surface with a Riemannian metric and let
$SM$ be its unit tangent bundle. 
Let
\[\cA:=\{A:SM\to \mathfrak{u}(n):\;\;V^{2}(A)=-A\}.\]
The set $\cA$ is identified with the set of all unitary
connections on the trivial bundle $M\times\C^n$. Indeed, a function
$A$ satisfying $V^{2}(A)+A=0$ is a function that locally can be written
as $A(x,y,\theta)=a(x,y)\cos\theta+b(x,y)\sin\theta$.

Recall from the introduction that $A$ is said to be cohomologically trivial
if there exists a smooth $u:SM\to U(n)$ such that $C(x,v,t)=u(\phi_{t}(x,v))u^{-1}(x,v)$.
Differentiating with respect to $t$ and setting $t=0$ this is equivalent to
\begin{equation}
X(u)+Au=0.
\label{eq:trans}
\end{equation}

Let $\cA_{0}$ be the set of all cohomologically trivial connections, that is,
the set of all $A\in \mathcal A$ such that there exists $u:SM\to U(n)$
for which (\ref{eq:trans}) holds.

Given a vector field $W$ in $SM$, let $G_{W}$ be the set
of all $u:SM\to U(n)$ such that $W(u)=0$, i.e. first integrals
of $W$. Note that $G_{V}$ is nothing but the group
of gauge transformations of the trivial bundle $M\times \C^n$.

We wish to understand $\cA_{0}/G_{V}$. Now let $\cH_{0}$ be the set
of all $f:SM\to\mathfrak{u}(n)$ such that
$$H(f)+VX(f)=[X(f),f]$$
and there is $u:SM\to U(n)$ such that
$f=u^{-1}V(u)$. It is easy to check that
$G_{X}$ acts on $\cH_{0}$ by $f\mapsto a^{-1}f\,a+a^{-1}V(a)$
where $a\in G_{X}$.

\begin{Theorem} There is a 1-1 correspondence between
$\cA_{0}/G_{V}$ and $\cH_{0}/G_{X}$.
\label{theorem:1-1}
\end{Theorem}

\begin{proof} Forward direction: a cohomologically trivial connection $A$
comes with a $u$ such that $X(u)+Au=0$. If we set
$f:=u^{-1}V(u)$, then $f\in \cH_{0}$, i.e., $f$ satisfies the PDE
$H(f)+VX(f)=[X(f),f]$. This is a calculation, exactly
as in the proof of Theorem B in \cite{P}, but for the reader's convenience
we explain the geometric origin of this equation. Using $u$ we may
define a connection on $SM$ gauge equivalent to $\pi^*A$ by setting
$B:=u^{-1}du+u^{-1}\pi^*A u$, where $\pi:SM\to M$ is the foot-point projection.
 Since $\pi^*A$ is the pull-back of a connection
on $M$, the curvature $F_{B}$ of $B$ must vanish when one of the entries
is the vertical vector field $V$. The PDE $H(f)+VX(f)=[X(f),f]$
arises by combining the two equations
$F_{B}(X,V)=F_{B}(H,V)=0$ with $B(X)=0$.

Backward direction: Given $f$ with $fu=V(u)$, set $A:=-X(u)u^{-1}$.
Then $A\in \cA_{0}$, i.e. $V^{2}(A)=-A$; again this is a calculation
done fully in Theorem B in \cite{P}.

Now there are two ambiguities here. Going forward, we may change
$u$ as long as we solve $X(u)+Au=0$. This changes $f$ by the action
of $G_{X}$. Going backwards we may change $u$ as long as $fu=V(u)$,
this changes $A$ by a gauge transformation, i.e. an element
in $G_{V}$.

\end{proof}

\begin{Remark}{\rm 
Note that if the geodesic flow is transitive (i.e. there is a dense orbit)
the only first integrals are the constants and thus $G_{X}=U(n)$
acts simply by conjugation. If $M$ is closed and of negative curvature,
the geodesic flow is Anosov and therefore transitive.

The fact that the PDE describing cohomologically trivial connections
arises from zero curvature conditions is an indication of the possible
``integrable'' nature of the problem at hand. The existence
of a B\"acklund transformation that we will introduce shortly is another
typical feature of integrable systems. An interesting point here
is that the space $\cH_{0}/G_{X}$ is in some sense simpler and larger when
the underlying geodesic flow is more complicated, i.e. when it is transitive
$G_X$ reduces to $U(n)$.

}

\end{Remark}

\section{The B\"acklund transformation}
\label{bt}
For the remainder of this paper we restrict to the case in which the structure group is
$SU(2)$.

Suppose there is a smooth map $b:SM\to SU(2)$ such that
$f:=b^{-1}V(b)$ solves the PDE:
\begin{equation}
H(f)+VX(f)=[X(f),f].
\label{mypde}
\end{equation}
Then, by Theorem \ref{theorem:1-1}, $A:=-X(b)b^{-1}$ defines a cohomologically trivial connection on $M$
and $-\star A=V(A)=-bX(f)b^{-1}-H(b)b^{-1}$. 

\begin{Lemma}
Let $g:M\to \mathfrak{su}(2)$ be a smooth map with
$\det g=1$ (i.e. $g^{2}=-\mbox{\rm Id}$). Then, there exists
$a:SM\to SU(2)$ such that $g=a^{-1}V(a)$.
\label{lemma:int}
\end{Lemma}

\begin{proof}

Let $L(x)$ (resp. $U(x)$) be the eigenspace corresponding to the eigenvalue $i$ (resp. $-i$) of $g(x)$. 
We have an orthogonal decomposition $\C^{2}=L(x)\oplus U(x)$ for every
$x\in M$.
Consider sections $\alpha\in \Omega^{1,0}(M,\C)$ and
$\beta\in \Omega^{1,0}(M,\mbox{\rm Hom}(L,U))=\Omega^{1,0}(M,L^*U)$ such that
$|\alpha|^2+|\beta|^2=1$. Such a pair of sections always exists; for example, we can choose a section $\tilde{\beta}$ with a finite number of isolated zeros
and then choose $\tilde{\alpha}$ such that it does not vanish on the zeros
of $\tilde{\beta}$. Then we
set $\alpha:=\tilde{\alpha}/(|\tilde{\alpha}|^2+|\tilde{\beta}|^2)^{1/2}$ and $\beta:=\tilde{\beta}/(|\tilde{\alpha}|^2+|\tilde{\beta}|^2)^{1/2}$.
Note that $\bar{\alpha}\in  \Omega^{0,1}(M,\C)$
and $\beta^*\in \Omega^{0,1}(M,\mbox{\rm Hom}(U,L))=\Omega^{0,1}(M,U^*L)$. 
Using the orthogonal decomposition we define $a:SM\to SU(2)$ by
 $$a(x,v)=\left(\begin{array}{cc}
\alpha(x,v)&\beta^*(x,v)\\
-\beta(x,v)&\bar{\alpha}(x,v)\\
\end{array}\right).$$
Clearly $a=a_{-1}+a_{1}$, where
 $$a_{1}=\left(\begin{array}{cc}
\alpha&0\\
-\beta&0\\
\end{array}\right)$$
and
 $$a_{-1}=\left(\begin{array}{cc}
0&\beta^*\\
0&\bar{\alpha}\\
\end{array}\right).$$
It is straightforward to check that $ag=V(a)$.

\end{proof}

Now let $u:=ab:SM\to SU(2)$ and let $F:=(ab)^{-1}V(ab)=b^{-1}g\,b+f$.

\medskip

\noindent{\bf Question.} When does $F$ satisfy (\ref{mypde})?

\medskip

If it does, then it defines (via Theorem \ref{theorem:1-1}) a new cohomologically trivial connection given
by
\[A_{F}=-X(ab)(ab)^{-1}=-X(a)a^{-1}+aAa^{-1},\]
where $A$ is the cohomologically trivial connection associated to
$f$. 

Recall that the connection $A$ defines a covariant derivative $d_{A}g=dg+[A,g]$.

\begin{Lemma} $F$ satisfies (\ref{mypde}) if and only if
\begin{equation}
-\star d_{A}g=(d_{A}g)\,g.
\label{gmero}
\end{equation}
\label{lemma:gmero}
\end{Lemma}

\begin{proof}

Starting  with $F=b^{-1}g\,b+f$
and using that $A=-X(b)b^{-1}=bX(b^{-1})$ we compute
\[X(F)=b^{-1}\left([A,g]+X(g)\right)b+X(f).\]
Similarly, using $H(b)=(\star A)b-bX(f)$ we find
\[H(F)=b^{-1}\left([-\star A,g]+H(g)\right)b+[X(f), b^{-1}g\,b]+H(f).\]
Now we compute $VX(F)$; here we use that $V(g)=0$.
We obtain
\[VX(F)=[b^{-1}([A,g]+X(g))b,f]+b^{-1}\left([-\star A,g]+VX(g)\right)b+VX(f).\]
The last term we need for (\ref{mypde}) is:

\[[X(F),F]=b^{-1}[[A,g]+X(g),g]b+[b^{-1}([A,g]+X(g))b,f]+
[X(f),b^{-1}g\,b]+[X(f),f].\]
Since $f$ satisfies (\ref{mypde}) we see that $F$ satisfies
(\ref{mypde}) if and only if
\[H(g)+VX(g)-2[\star A,g]=[[A,g]+X(g),g].\]
Since $g$ depends only on the base point and $g^{2}=-\mbox{\rm Id}$ we can rewrite this
as
\[-2\star(dg+[A,g])=[dg+[A,g],g]=2(dg+[A,g])\,g.\]
Thus $F$ satisfies (\ref{mypde}) if and only
if
\[-\star d_{A}g=(d_{A}g)\,g\]
as claimed.

\end{proof}

We will now rephrase equation (\ref{gmero}) in terms of holomorphic line bundles.
Recall that the connection $A$ induces a holomorphic structure on the trivial bundle $M\times \C^2$
and on the endomorphism bundle $M\times \M_{2}(\C)$. 
We have an operator $\bar{\partial}_{A}=(d_{A}-i\star d_{A})/2=\bar{\partial}+[A_{-1},\cdot]$
acting on sections $f:M\to \M_{2}(\C)$.
 
Set $\pi:=(\mbox{\rm Id}-ig)/2$ and $\pi^{\perp}=(\mbox{\rm Id}+ig)/2$ so that
$\pi+\pi^{\perp}=\mbox{\rm Id}$.
Let $L(x)$ be as above the eigenspace corresponding 
to the eigenvalue $i$ of $g(x)$. Note that $\pi$ is the Hermitian orthogonal projection
over $L(x)=\mbox{\rm Image}(\pi(x))$.

\begin{Lemma} Let $g:M\to \mathfrak{su}(2)$ be a smooth map with $\det g=1$.
The following are equivalent:
\begin{enumerate}
\item $-\star d_{A}g=(d_{A}g)g$;
\item $L$ is a $\bar{\partial}_{A}$-holomorphic line bundle;
\item $\pi^{\perp}\bar{\partial}_{A}\pi=0$.
\end{enumerate}
\label{lemma:eq}
\end{Lemma}

\begin{proof}

Suppose that (1) holds. Apply $\star$ to obtain: $d_{A}g=(\star d_{A}g)\,g$.
Thus $d_{A}g-i\star d_{A}g=i(d_{A}g-i\star d_{A}g)g$.
In other words $\bar{\partial}_{A}g=i(\bar{\partial}_{A}g)\,g=-ig(\bar{\partial}_{A}g)$ (recall
that $g^2=-\mbox{\rm Id}$).
Since $\pi=(\mbox{\rm Id}-ig)/2$, then $\bar{\partial}_{A}g=-ig(\bar{\partial}_{A}g)$ 
is equivalent to $\pi^{\perp}\bar{\partial}_{A}\pi=0$ which is (3).

Using the condition $\pi^2=\pi$, we see that $\pi^{\perp}\bar{\partial}_{A}\pi=0$ is equivalent
to $(\bar{\partial}_{A}\pi)\pi=0$. The line bundle $L$ is holomorphic iff given
a local section $\xi$ of $L$, then $\bar{\partial}_{A}\xi\in L$. Using that $\pi\xi=\xi$ we see that
$\bar{\partial}_{A}\xi\in L$ iff $(\bar{\partial}_{A}\pi)\xi=0$.
Clearly, this happens iff $(\bar{\partial}_{A}\pi)\pi=0$ and thus (2) holds iff (3) holds.

\end{proof}

The next theorem summarises the B\"acklund transformation introduced in this section and it follows directly from Lemmas \ref{lemma:gmero} and \ref{lemma:eq} and Theorem \ref{theorem:1-1}.

\begin{Theorem} Let $A$ be a cohomologically trivial connection and let $L$ be
a holomorphic line subbundle of the trivial bundle $M\times\mathbb C^2$ with respect to the complex structure induced by $A$. Define a map
$g:M\to \mathfrak{su}(2)$ with $\det g=1$ by declaring $L$ to be its eigenspace with
eigenvalue $i$. Consider $a:SM\to SU(2)$ with $g=a^{-1}V(a)$ as given by Lemma \ref{lemma:int}. Then
\[A_{F}:=-X(a)a^{-1}+aAa^{-1}\]
defines a cohomologically trivial connection.
\label{cor:back}

\end{Theorem}

\begin{Remark}{\rm Note that if the geodesic flow is transitive, two solutions
$u,w$ of $X(u)+Au=0$ are related by $u=wg$ where $g$ is a constant unitary
matrix, because $X(w^{-1}u)=0$. Thus the degrees of $u$ and $w$ are the same. We can then talk
about the ``degree'' of a cohomologically trivial connection as the degree
of any solution of $X(u)+Au=0$.
}
\end{Remark}

If we start, for example, with the trivial connection $A=0$ (which is obviously transparent), then
a map $g:M\to\mathfrak{su}(2)$ with $\det g=1$ and $-\star dg=(dg)g$ can be identified with a meromorphic
function. The connections of degree one $A_{F}=-X(a)a^{-1}$ given by Theorem \ref{cor:back} were found
in \cite{P}. To some extent, the point of this paper is to show that one can continue this process and
also run it ``backwards'' to decrease the degree of a transparent connection.
In the next section we will show that any cohomologically trivial connection such that the associated $u$ has
a finite Fourier series can be built up by successive applications of the transformation
described in Theorem \ref{cor:back}, provided that the geodesic flow is transitive.

\section{Lowering degree using B\"acklund transformations}

Let $A$ be a transparent connection with $A=-X(b)b^{-1}$ and
$f=b^{-1}V(b)$, where $b:SM\to SU(2)$, is as in the previous section.

We first make some remarks concerning the $SU(2)$-structure. Let
$j:\mathbb C^2\to\mathbb C^2$ be the antilinear map given by
\[j(z_1,z_2)=(-\bar{z}_{2},\bar{z}_{1}).\]
If we think of a matrix $a\in SU(2)$ as a linear map $a:\mathbb C^2\to\mathbb C^2$,
then $ja=aj$. This implies that given $b:SM\to SU(2)$ with
$b=\sum_{k\in\mathbb Z}b_k$, then $jb_{k}=b_{-k}j$ for all $k\in\mathbb Z$.

\medskip

\noindent{\bf Assumption.} Suppose $b$ has a finite Fourier
expansion, i.e., $b=\sum_{k=-N}^{k=N}b_{k}$, where $N\geq 1$.
By Theorem \ref{theorem:finite} we know that this holds if $M$ has negative curvature.

\medskip
Let us assume also that $N$ is the degree of $b$ and thus both $b_{N}$ and $b_{-N}=-jb_{N}j$ are
non-zero.

The unitary condition $bb^*=b^* b=\mbox{\rm Id}$ implies that 
$b_{N}b^{*}_{-N}=b^{*}_{-N}b_{N}=0$. These relations imply that
the rank of $b_{-N}$ and $b_{N}$ is at most one and equals
one on an open set, which, as we will see shortly, must be all of $M$ except for perhaps
a finite number of points.

 Consider now a fixed vector $\xi\in \mathbb C^2$ such
that $s(x,v):=b_{-N}(x,v)\xi\in \mathbb C^2$ is not zero identically.
Clearly $s$ can be seen as a section of $(M\times \mathbb C^2)\otimes K^{\otimes -N}$.
We may write $b_{-N}$ in local isothermal coordinates as $b_{-N}=he^{-iN\theta}$,
using the notation from
Section \ref{prelim}.
We can thus write $s$ locally as $s=e^{N\lambda}h\xi(d\bar{z})^N$.

\begin{Lemma} The local section $e^{-2N\lambda}s$ is $\bar{\partial}_{A}$-holomorphic.
\end{Lemma}

\begin{proof}Using the operators $\mu_{\pm}$ introduced in Section \ref{prelim} we can write $X(b)+Ab=0$ as
\[\mu_{+}(b_{k-1})+\mu_{-}(b_{k+1})=0\]
for all $k$. This gives $\mu_{+}(b_{N})=\mu_{-}(b_{-N})=0$.
But $\mu_{-}(b_{-N})=0$ is saying that $e^{-2N\lambda}s$ is $\bar{\partial}_{A}$-holomorphic. 
Indeed, using (\ref{eq:mu}), we see that $\mu_{-}(b_{-N})=0$ implies
\[ \bar{\partial}(he^{-N\lambda})+A_{\bar{z}}he^{-N\lambda}=0\]
which in turn implies
\[ \bar{\partial}(e^{-N\lambda}h\xi)+A_{\bar{z}}e^{-N\lambda}h\xi=0.\]
This equation says that $e^{-2N\lambda}s=e^{-N\lambda}h\xi(d\bar{z})^N$ is $\bar{\partial}_{A}$-holomorphic.

\end{proof}

The section $s$ spans a line bundle $L$ over $M$ which by the previous lemma is 
$\bar{\partial}_{A}$-holomorphic. The section $s$ may have zeros, but at a zero $z_0$, the line bundle extends holomorphically. Indeed, in a neighbourhood
of $z_0$ we may write
$e^{-2N\lambda(z)}s(z)=(z-z_0)^k w(z)$, where $w$ is a local holomorphic
section with $w(z_0)\neq 0$. The section $w$ spans a holomorphic line subbundle which coincides with the one spanned by $s$ off $z_0$.
Therefore $L$ is a $\bar{\partial}_{A}$-holomorphic line bundle that contains the image of $b_{-N}$
(and $U=jL$ is an anti-holomorphic line bundle that contains the image of $b_{N}$).
We summarise this in a lemma:

\begin{Lemma} The line bundle $L$ determined by the image of $b_{-N}$
is $\bar{\partial}_{A}$-holomorphic.
\label{lemma:holo}
\end{Lemma}

We now wish to use the line bundle $L$ to construct an appropriate $g:M\to\mathfrak{su}(2)$
such that when we run the B\"acklund transformation from the previous section
we obtain a cohomologically trivial connection of degree $\leq N-1$.
But first we need the following lemma. Recall that a matrix-valued
function $f:SM\to \M_{n}(\mathbb C)$ is said to be {\it odd} if
$f(x,v)=-f(x,-v)$ and {\it even} if $f(x,v)=f(x,-v)$.

\begin{Lemma} Assume that the geodesic flow is transitive and let
$b:SM\to SU(2)$ solve $X(b)+Ab=0$. Then $b$ is either even or odd.
\label{lemma:oe}
\end{Lemma}

\begin{proof} Write $b=b_{o}+b_{e}$ where $b_o$ is odd and $b_e$ is even.
Since the operator $(X+A)$ maps even to odd and odd to even, the
equation $X(b)+Ab=0$ decouples as 
\[X(b_o)+Ab_o=0;\]
\[X(b_e)+Ab_e=0.\]
A calculation using these equations shows that $X(b_o^*b_o)=X(b_e^*b_e)=X(b_{o}^{*}b_{e})=0$. Since
the geodesic flow is transitive, these matrices are all constant. Moreover, since
$b_o^*b_e$ is odd it must be zero. On the other hand $jb=bj$ implies
that $jb_o=b_o j$ and $jb_e=b_e j$, which in turn implies that both $b_o$ and $b_e$
cannot have rank 1. Putting all this together, we see that either $b_o$ or $b_e$
must vanish identically.

\end{proof}

Suppose that the geodesic flow is transitive. By Lemma \ref{lemma:oe},
$b=b_{-N}+d+b_{N}$, where
$d$ has degree $\leq N-2$. We now
seek $a:SM\to SU(2)$ of degree one such that $u:=ab$ has degree
$\leq N-1$. For this we need $a_{1}b_{N}=a_{-1}b_{-N}=0$.
We take a map $g:M\to \mathfrak{su}(2)$ with $\det g=1$ such
that its $i$ eigenspace is $L$ and its $-i$ eigenspace is $U$.
By Lemmas \ref{lemma:eq} and \ref{lemma:holo}, $-\star d_{A}g=(d_{A}g)\,g$. The construction
of $a$ with $ag=V(a)$ from Lemma \ref{lemma:int} is precisely such that
the kernel of $a_{-1}$ is $L$ and the kernel of $a_{1}$ is $U$, so the
needed relations to lower the degree hold.

Finally by Theorem \ref{cor:back}, $u$ gives rise to a cohomologically trivial connection
$-X(u)u^{-1}$. Combining this with Theorem \ref{theorem:finite} we have 
arrived at the main result of this paper:

\begin{Theorem} Let $M$ be a closed orientable surface of negative curvature.
Then any transparent $SU(2)$-connection can be obtained by successive applications
of B\"acklund transformations as described in Theorem \ref{cor:back}.
\label{theorem:main}
\end{Theorem}

\begin{Remark}{\rm It should be possible to do a similar analysis when
 the structure group is $U(n)$ by considering maps $g:M\to G_{k,n}$
where $G_{k,n}$ is the complex Grassmannian of $k$-planes in $\mathbb C^n$.
There is a natural embedding of $G_{k,n}$ into $U(n)$ defined by sending
a $k$-plane $\ell\in G_{k,n}$ into the unitary map $i(\pi_{\ell}-\pi_{\ell^{\perp}})$, where $\pi_{\ell}$ denotes the Hermitian orthogonal projection
onto $\ell$.
It would also be interesting to understand the set of transparent 
connections for the case of non-trivial bundles and for structure groups
other than $U(n)$.

}
\end{Remark}

\end{document}